\documentclass[10pt]{article}

\usepackage{amsmath,amssymb,amsthm,graphicx,mathrsfs,url}
\usepackage[usenames,dvipsnames]{color}

\definecolor{darkred}{rgb}{0.4,0.1,0.1}
\usepackage[colorlinks=true,linkcolor=darkred,citecolor=Blue]{hyperref}

\DeclareMathOperator\ran{ran\,}
\DeclareMathOperator\rank{rank\,}

      \def\sI{{\mathfrak I}}

      \def\dC{{\mathbb C}}

   \def\dN{{\mathbb N}}

      \def\cL{{\mathcal L}}

\newcommand{\dom}{\mathrm{dom}\,}

\numberwithin{equation}{section}
\theoremstyle{plain}% default
\newtheorem{thm}{Theorem}[section]
\newtheorem{hypothesis}[thm]{Hypothesis}

\newtheorem{exam}[thm]{Example}
\newtheorem{lem}[thm]{Lemma}

\newtheorem{prop}[thm]{Proposition}
\newtheorem{cor}[thm]{Corollary}

\theoremstyle{plain}

\numberwithin{equation}{section}

\begin{document}

\title{\textbf{The effect of finite rank perturbations on Jordan chains of linear operators}}
%[Finite rank perturbations and Jordan chains]

\author{Jussi Behrndt, Leslie Leben, Francisco Mart\'{i}nez Per\'{i}a, \\ and Carsten Trunk}
%\author[J. Behrndt]{Jussi Behrndt}
%\address{Institut f\"ur Numerische Mathematik, Technische Universit\"at Graz, Steyrergasse 30, 8010 Graz, Austria}
%\email{behrndt@tugraz.at}
%\thanks{Jussi Behrndt gratefully acknowledges financial support by the Austrian
%Science Fund (FWF), project P 25162-N26.}

%    author two information
%\author[L. Leben]{Leslie Leben}
%\address{Institut f\"ur  Mathematik,  Technische Universit\"{a}t Ilmenau, Postfach 100565, D-98684 Ilmenau,  Germany}
%\email{leslie.leben@tu-ilmenau.de}
%\thanks{Leslie Leben gratefully acknowledges the support from the Carl-Zeiss-Stiftung Jena.}

%\author[F. Mart\'{\i}nez Per\'{\i}a]{Francisco Mart\'{\i}nez Per\'{\i}a}
%\address{Departamento de Matem\'{a}tica -- Facultad de Ciencias Exactas, Universidad Nacional de La Plata, C.C.\ 172, (1900) La Plata, Argentina \\ and Instituto Argentino de Matem\'{a}tica "Alberto P. Calder\'{o}n" (CONICET), Saavedra 15 (1083) Buenos Aires, Argentina }
%\email{francisco@mate.unlp.edu.ar}
%\thanks{Francisco Mart\'{\i}nez Per\'{\i}a gratefully acknowledges the support from the grant  PIP CONICET 0435}

%\author[C. Trunk]{Carsten Trunk}
%\address{Institut f\"ur  Mathematik,  Technische Universit\"{a}t Ilmenau, Postfach 100565, D-98684 Ilmenau,  Germany}
%\email{carsten.trunk@tu-ilmenau.de}

%    \subjclass is required.
%\subjclass[2010]{Primary 47A55; Secondary 47A10, 15A18}

\date{}

\maketitle

%\dedicatory{}

%    "Communicated by" -- provide editor's name; required.
%\commby{}

%    Abstract is required.
\begin{abstract}
\noindent
A general result on the structure and dimension of the root subspaces of a linear operator under finite rank perturbations is proved:
The increase of dimension from the $n$-th power of the kernel of the perturbed operator to the $(n+1)$-th power
differs from the increase of dimension of the corresponding powers of the kernels of the unperturbed operator by at most
the rank of the perturbation. This bound is sharp.
\end{abstract}

\bigskip
%\noindent

\section{Introduction}

Perturbation theory for linear operators and their spectra is
one of the main objectives in operator theory and functional analysis, with numerous
applications in mathematics, physics and engineering sciences.
 In many approaches compact perturbations and perturbations small in size
are investigated, e.g.\ when stability properties of the index,
 nullity and deficiency of Fredholm and semi-Fredholm operators are analysed.
A widely used and well-known fact on the effect of compact
perturbations is the following: If $S$ and $T$ are bounded operators
in a Banach space, $K=S-T$ is compact and $\lambda\in\dC$ is
such that $S-\lambda$ is Fredholm, then also $T-\lambda$ is Fredholm and the Fredholm
index is preserved. In particular, since $\ker (S-\lambda)$ and $\ker(S-\lambda)^{n+1} / \ker(S-\lambda)^n$
are finite dimensional the same is true for
$\ker (T-\lambda)$ and $\ker(T-\lambda)^{n+1} / \ker(T-\lambda)^n$. However,
for such an arbitrary compact perturbation $K$ there exists no
bound %in $\lambda$
on the
dimensions of $\ker (T-\lambda)$ or $\ker(T-\lambda)^{n+1} / \ker(T-\lambda)^n$.
The situation is different when the perturbation is not only compact
but of finite rank.

In the present note we consider general linear operators $S$ and $T$ in a vector space $X$
such that $T$ is a finite rank perturbation of $S$.
%under finite rank perturbations.
It follows easily that
the dimensions of $\ker(S-\lambda)$ and $\ker(T-\lambda)$ differ
at most by $k$
if the perturbation $K=S-T$ is an operator with $\rank (K)=k$.
Our main objective is to explore the connections between the kernels of consecutive higher
powers of $S-\lambda$ and $T-\lambda$ in more detail, and to prove the following general result on the structure and dimensions of the
root subspaces under finite rank perturbations: Given a linear operator $S$ in $X$,
consider the space $\ker(S -\lambda)^{n+1}/\ker(S-\lambda)^n$.
Its dimension coincides with the number of linearly independent Jordan chains of $S$ at $\lambda$ of length at least $n+1$.
It then turns out that the change of the number of these Jordan chains of $S$ at $\lambda$
under a rank $k$ perturbation is bounded by $k$,
\begin{equation}\label{PudelsKern}
\left| \dim \left(\frac{\ker (S-\lambda)^{n+1}}{\ker (S-\lambda)^{n}}\right)-
\dim \left(\frac{\ker (T-\lambda)^{n+1}}{\ker (T-\lambda)^{n}} \right) \right| \leq k,
\end{equation}
and this bound is sharp, see Theorem~\ref{hurra} and Example~\ref{sharp}. Here $S$ and $T$ are defined on subspaces of $X$ and the finite rank perturbation is interpreted in a generalized sense, see Hypothesis \ref{hypo}.
In particular, our assumptions allow to treat
unbounded operators in Banach spaces and finite rank perturbations in resolvent sense.
We also emphasize that the dimensions of the root subspaces of the operators $S$ and $T$ may be infinite,
and that a finite rank perturbation may turn points from the resolvent set of $S$ into eigenvalues of infinite algebraic
multiplicity of $T$; cf. Example~\ref{examinf}.

If $X$ is finite dimensional, then $S$ and $T$ are matrices and \eqref{PudelsKern} was already proved by S.V.\ Savchenko in \cite[Lemma 2]{S05}, see also
\cite{BHZ,DM02,MMRR11,MMRR12,MMRR13,S03,S04} for related results on so-called generic perturbations of matrices.
Moreover, there exists a lower bound for the dimension of the root subspace of the perturbed operator $T$ in terms of the dimension of the root subspace of $S$ and the length of the Jordan chains of $S$ at $\lambda$;
cf. \cite{DM03, S05}. Such a
result was also proved by L.~H\"ormander and A.~Melin in a more general case:
the unperturbed operator $S$ is compact and the perturbation $K=T-S$ is of finite rank, see \cite[Theorem~3]{HM94}.
In Corollary \ref{corHurra} we obtain the same bound for the general setting considered here.

{\bf Acknowledgement.} The authors are most grateful and deeply indebted to Marinus A.\ Kaashoek, Volker Mehrmann,
Leiba Rodman, and Sergey V.\ Savchenko for fruitful discussions and very useful literature hints.

Jussi Behrndt gratefully acknowledges financial support by the Austrian Science Fund (FWF), project P 25162-N26.
Leslie Leben gratefully acknowledges the support from the Carl-Zeiss-Stiftung Jena.
Francisco Mart\'{\i}nez Per\'{\i}a gratefully acknowledges the support from CONICET PIP 0435.

\section{Main result}

Let $X$ be a vector  space over $\mathbb K$, where $\mathbb K$ stands
 either for $\mathbb R$ or $\mathbb C$.
 Let $S$ and $T$ be linear
  operators in $X$ defined on some linear subspaces
$\dom S$ and $\dom T$ of $X$, respectively.
We consider finite rank perturbations in the following generalized sense:

\begin{hypothesis}\label{hypo}
 There exists a linear subspace $M$ contained in $\dom S\cap\dom T$
 such that the restrictions $S\upharpoonright M$ and $T\upharpoonright M$ coincide on $M$ and
 $$
 \max\bigl\{\dim (\dom S /M),\dim (\dom T / M)\bigr\}=k<\infty.
 $$
\end{hypothesis}

Three typical situations where the above hypothesis is satisfied are the following:
\begin{itemize}
 \item [{\rm (i)}] $X$ is a finite dimensional space,
 $S$ and $T$ are defined on $X$ and the rank of $S-T$ is $k$. In this case, for a
  fixed basis of $X$, $S$ and $T$ are represented
 by matrices.
 \item [{\rm (ii)}] If $X$ is an arbitrary vector space, $\dom S= \dom T$ and $$\dim(\ran(S-T))=k.$$
 \item [{\rm (iii)}] $X$ is a Banach space, $S$ and $T$ are densely defined closed operators in $X$, and
there exists  $\mu\in\mathbb K$ in the resolvent set of $S$ and $T$ with
  $$
 \dim\bigl(\ran\bigl((S-\mu)^{-1}-(T-\mu)^{-1}\bigr)\bigr)=k.
 $$
\end{itemize}

Given $\lambda\in\mathbb K$, a finite ordered set of non-zero vectors
$\{x_0, \dots, x_{n-1}\}$ in $\dom S$ is a \emph{Jordan chain of length} $n$ at $\lambda$
if $(S-\lambda)x_0 = 0$
and $(S-\lambda)x_i = x_{i-1}$, $i = 1, \dots,n-1$. A Jordan chain of infinite length is defined accordingly.
 The elements of a Jordan chain are linearly independent.
The first $n-1$ elements of a Jordan chain of length $n$ form a Jordan chain of length $n-1$.
%Let $\{x_0, \dots, x_{n-1}\}$ and $\{y_0, \dots, y_{m-1}\}$ be Jordan chains of $S$ at $\lambda$. Then
%\begin{equation}\label{111}
% \{x_0, \dots, x_{n-1}\}\,\,\text{and}\,\,\{y_0, \dots, y_{m-1}\}\,\,\text{are linearly independent}
%\end{equation}
%if and only if the vectors
%\begin{equation}\label{222}
% x_0\,\,\text{and}\,\, y_0\,\,\text{are linearly independent.}
%\end{equation}
%The equivalence of \eqref{111} and \eqref{222} will be used in Section~\ref{section3}.
Furthermore, we say that $S$ has $k$ Jordan chains of length $n$ at $\lambda$
if there exist $k$ linearly independent Jordan chains of length $n$.
The \emph{root subspace $\cL_\lambda(S)$ of $S$
at $\lambda$} is the collection of all
Jordan chains of $S$ at $\lambda$,
$$\cL_\lambda(S)=\bigcup_{j=1}^\infty\ker(S-\lambda)^j.$$

The following theorem is the main result of this article.
In the special case that $X$ is finite dimensional %and $S$ and $T$ are matrices
it coincides with \cite[Lemma 2]{S05}.
The proof of Theorem~\ref{hurra} is given in Section \ref{proof}.

\begin{thm}\label{hurra}
 Let $S$ and $T$ be linear operators in $X$ satisfying Hypothesis~{\rm \ref{hypo}}.
 Then, the following holds for every $\lambda\in\mathbb K$:
  \begin{itemize}
 \item[\rm (i)] If $\ker( S-\lambda)^n$ is finite dimensional
  for some $n\in \mathbb N$, then the same holds
for $\ker (T-\lambda)^n$ and
\begin{align}\label{formel1}
\vert \dim \ker (S-\lambda)^n &- \dim \ker (T-\lambda)^n \vert \leq k\,n.
\end{align}
  \item[\rm (ii)] If $\ker (S-\lambda)^{n+1} / \ker (S-\lambda)^{n}$ is
  finite dimensional for some $n\in \mathbb N$, then the same holds
for $\ker (T-\lambda)^{n+1} / \ker (T-\lambda)^{n}$ and
\begin{equation}\label{formel2}
\left| \dim \left(\frac{\ker (S-\lambda)^{n+1}}{\ker (S-\lambda)^{n}}\right)-
\dim \left(\frac{\ker (T-\lambda)^{n+1}}{\ker (T-\lambda)^{n}} \right) \right| \leq k.
\end{equation}
 \end{itemize}
\end{thm}

The estimates in Theorem~\ref{hurra} are sharp
in the following sense.

\begin{exam}\label{sharp}
 In $X=\mathbb{K}^m$ consider a fixed basis $\{e_1,\dots,e_m\}$ and,
 with respect to this basis,  let the linear operators $A_1$ and $B_1$ be given
  via their $m\times m$ matrix-representation
\begin{equation*}
A_1=\begin{pmatrix}
     0 & 1 & 0 &  \cdots & 0\\
     0 & 0 & 1 &  \cdots  & 0\\
     \vdots & \vdots  & & \ddots & \vdots\\
     0 & 0 & 0  & \cdots & 1 \\
     0 & 0 &  0 &\cdots & 0
    \end{pmatrix}\quad\text{and}\quad
B_1=\begin{pmatrix}
     0 & 1 & 0 &  \cdots & 0\\
     0 & 0 & 1 &  \cdots  & 0\\
     \vdots & \vdots  & & \ddots & \vdots\\
     0 & 0 & 0  & \cdots & 1 \\
     1 & 0 &  0 &\cdots & 0
    \end{pmatrix}.
\end{equation*}
Then $A_1$ and $B_1$ satisfy Hypothesis~{\rm \ref{hypo}} with $k=1$ and $M=\text{\rm span}\,\{e_2,\dots,e_m\}$, and we have
for $j\leq m$
$$
\ker A_1^j=\text{\rm span}\,\bigl\{e_1,\dots,e_j\bigr\}\quad\text{and}\quad \ker B_1^j=\{0\}.
$$
 Hence the assertions in Theorem~{\rm \ref{hurra}} are sharp
 for the case $\lambda=0$ and $k=1$.
In order to obtain sharpness for general $k\in\dN$
consider the $(mk\times mk)$-matrices in $X^k$,
$$
A=A_1\oplus \dots \oplus A_1\quad\text{and}\quad B=B_1\oplus \dots \oplus B_1.
$$
\end{exam}

In the following corollary the bounds in Theorem~\ref{hurra} are considered in the context of the dimensions
of the root subspaces.

\begin{cor}\label{corcor}
 Let $S$ and $T$ be linear operators in $X$ satisfying Hypothesis~{\rm \ref{hypo}}.
 Assume that the root subspace  $\cL_\lambda(S)$
 of $S$ at $\lambda\in\mathbb K$
 is finite dimensional. Then, the following holds:
  \begin{itemize}
 \item[\rm (i)] If the maximal length of Jordan chains of $S$ at $\lambda$  is bounded by $p$ then
\begin{align*}
\vert \dim \cL_\lambda(S) &- \dim \ker (T-\lambda)^p \vert \leq k\,p.
\end{align*}
  \item[\rm (ii)]
   If the maximal lengths of Jordan chains of $S$ at $\lambda$ and
   Jordan chains of $T$ at $\lambda$ are bounded by $p$ and $q$, respectively, then $\cL_\lambda(T)$ is finite dimensional and
\begin{align*}
\bigl| \dim  \cL_\lambda(S)&-
\dim  \cL_\lambda(T)\bigr| \leq k\,\max\{p,q\}.
\end{align*}
 \end{itemize}
\end{cor}
\begin{proof}
In (i) we have $\cL_\lambda(S) = \ker (S-\lambda)^p$.
In (ii) we have, in addition,  $\cL_\lambda(T) = \ker (T-\lambda)^q$.
Then (i) and (ii) follow from \eqref{formel1}.
\end{proof}

We emphasize that in (i) of Corollary~\ref{corcor} (where it is assumed that $\cL_\lambda(S)$ is finite dimensional)
the root subspace $\cL_\lambda(T)$
may be infinite dimensional. This will be illustrated by the following example,
where a rank one perturbation of a bijective operator generates an infinitely long
Jordan chain.

\begin{exam}\label{examinf}
Let $X=\ell^2(\dN)\times \ell^2(\dN)$ and consider the following operators $S$ and $T$ in $X$:
\begin{equation*}
\begin{split}
 S\begin{pmatrix}(x_n)_{n\in\dN}\\(y_n)_{n\in\dN}\end{pmatrix}&:=\begin{pmatrix}(y_1,x_1,x_2,\dots)\\ (y_2,y_3,y_4,\dots)\end{pmatrix},\\
 T\begin{pmatrix}(x_n)_{n\in\dN}\\(y_n)_{n\in\dN}\end{pmatrix}&:=\begin{pmatrix}(0,x_1,x_2,\dots)\\ (y_2,y_3,y_4,\dots)\end{pmatrix}.
\end{split}
\end{equation*}
It is clear that the operator $S-T$ is of rank one, and %it is easy to see that $S$ is bijective. In particular,
\begin{equation*}
 \ker S=\{0\}.%\quad\text{and}\quad\dim\cL_0(S)=0.
\end{equation*}
On the other hand $T$ has a Jordan chain at $0$ of infinite length, which is given by
$\{\left( \begin{smallmatrix} 0\\ e_n \end{smallmatrix}\right) : n\geq 1\}$ with $\{e_n : n \geq 1\}$
denoting the standard basis in $\ell^2$. Hence,
\begin{equation*}
 \dim\ker T^p= p\quad\text{and}\quad\dim\cL_0(T)=\infty.
\end{equation*}
\end{exam}

The bound in Corollary~\ref{corcor}~(ii) can be improved if the 
number $k$ from Hypothesis~{\rm \ref{hypo}}
is small compared to the number of
linearly independent Jordan chains of $S$. The following corollary was obtained in  \cite{DM03, S05} for matrices and
in \cite[Theorem 3]{HM94} for compact operators. The proof of Corollary~\ref{corHurra} below is omitted since it follows
the same arguments as the proof of \cite[Corollary 1]{S05}.

\begin{cor}\label{corHurra}
 Let $S$ and $T$ be linear operators in $X$ satisfying Hypothesis~{\rm \ref{hypo}}.
 Assume that the root subspace  $\cL_\lambda(S)$
 of $S$ at $\lambda\in\mathbb K$
 is finite dimensional and let $n_1 \geq n_2 \geq \dots \geq n_l$ be the lengths
of the linearly independent Jordan chains of $S$ at $\lambda$. Then, for $k \leq l$
the following holds:
\begin{align}\label{savchenko}
\dim \cL_{\lambda}(S) - n_1 -  n_2 - \dots - n_k\leq \dim \cL_{\lambda}(T).
\end{align}
\end{cor}

We mention that in the situation of Corollary~\ref{corHurra} the root subspace $\cL_{\lambda}(T)$ may be infinite dimensional
(see Example~\ref{examinf}), and, in this case, the right hand side of \eqref{savchenko} is $\infty$.

\section{Preparatory statements}\label{section3}

In this section we prove Theorem~\ref{hurra} for the special case $k=1$.
Notice that it suffices to prove the result for $\lambda=0$; otherwise replace $S$ and $T$ by $S-\lambda$ and $T-\lambda$.
Theorem~\ref{hurra} in this  situation is formulated below in Proposition~\ref{general rank one pert}.
As a preparation we state two simple lemmas. The first
is an immediate consequence of the fact that $S$ and $T$ coincide
on the subspace $M$; cf. Hypothesis~{\rm \ref{hypo}}.

\begin{lem}\label{Mchain}
Let $S$ and $T$ be linear operators in $X$ satisfying Hypothesis~{\rm \ref{hypo}}.
If $\{x_0,\ldots,x_n\}$ is a Jordan chain of $S$ at $\lambda$ such that
 $x_k\in M$ for every $k=0,\ldots,n$, then $\{x_0,\ldots,x_n\}$ is
  also a Jordan chain of $T$ at $\lambda$.
\end{lem}

The next lemma follows from the fact that for a linear operator $A$ in $X$ the mapping $x + \ker A\mapsto Ax$,
is an isomorphism between  $X/\ker A$ and $\ran A$.

\begin{lem}\label{isom}
For a linear operator $A$ in $X$ the set $\{x_1 + \ker A,\ldots, x_m + \ker A\}$ is linearly independent
in $X/\ker A$ if and only if the set $\{Ax_1,\ldots, Ax_m\}$ is
 linearly independent in $X$.
\end{lem}

The next proposition is Theorem~\ref{hurra} in the special case $k=1$ and $\lambda = 0$.

\begin{prop}\label{general rank one pert}
Let $S$ and $T$ be linear operators in $X$ satisfying Hypothesis~{\rm \ref{hypo}} with
 $k=1$. Then the following holds:
 \begin{itemize}
 \item[\rm (i)] If $\ker S^n$ is finite dimensional for some $n\in \mathbb N$,
 $n\geq 1$, then the same holds
for $\ker T^n$ and
\begin{align}\label{formel11}
\vert \dim \ker S^n &- \dim \ker T^n \vert \leq n.
\end{align}
  \item[\rm (ii)] If $\ker S^{n+1} / \ker S^{n}$ is finite dimensional
   for some $n\in \mathbb N$, $n\geq 1$, then the same holds
for $\ker T^{n+1} / \ker T^{n}$ and
\begin{align}\label{formel22}
\bigl| \dim \left( \ker S^{n+1} / \ker S^{n} \right) &-
\dim \left( \ker T^{n+1} / \ker T^{n} \right) \bigr| \leq 1. \end{align}
 \end{itemize}
\end{prop}

\begin{proof}
First, we show (i) for the case $n=1$, i.e.
\begin{equation}\label{RegenErfurt}
\vert \dim\ker S - \dim \ker T \vert \leq 1.
\end{equation}
Assume that $\ker S$ is finite dimensional
and $\dim \ker T > \dim \ker S +1$. Then there exist $m:=  \dim \ker S +2$ linearly independent
vectors  $\{x_1, \dots, x_{m}\}$ in $\ker T$.
 If $x_j \in M$ then $Sx_j = Tx_j$. So, if $x_j \in M$ for
all $j=1,\ldots,m$ then $\{x_1, \dots, x_{m}\}\subseteq \ker S$,
a contradiction.

 Hence, there exists $1\leq k_0\leq m$ such that $x_{k_0} \in \ker T \setminus M$.
  After reordering we can assume that $k_0=m$.
As $\dim(\dom T/M)\leq 1$ it is easy to see that there exist $\alpha_k \in \mathbb K$ such that
$$
z_k:=x_k - \alpha_k x_m \in M,\qquad k=1,\ldots, m-1.
$$
Thus $S z_k = T z_k = 0$ for $k=1,\ldots, m-1,$ and we conclude that
$\{z_1,\ldots, z_{m-1}\}$ is a linearly independent set in $\ker S$; a contradiction.
Therefore, $\dim \ker T \leq \dim \ker S +1$ and, in particular,   $\ker T$
is finite dimensional.
By interchanging $S$ and $T$ we also obtain
$\dim \ker S -1 \leq \dim \ker T$ and hence
\eqref{RegenErfurt} follows.\\

In the following we prove
{\rm (ii)}. Let $n\in \mathbb N$, $n\geq 1$, such that
$\ker S^{n+1} / \ker S^{n}$ is finite dimensional and
set
\begin{equation}\label{mmm}
m:= \dim(\ker S^{n+1} / \ker S^{n}) +2.
\end{equation}
Assume that the set $\{ x_{1,n}+\ker T^n, \ldots, x_{m,n}+\ker T^n\}$ is linearly independent in
 $\ker T^{n+1} / \ker T^n$.
For $k=1,\ldots, m$ construct the following Jordan chains of $T$ at $0$:
\[
x_{k,0}:=T^n x_{k,n}, \quad x_{k,1}:=T^{n-1} x_{k,n},
\quad \ldots , \quad
x_{k,n-1}:= Tx_{k,n}.
\]
Then, $x_{k,0}\in\ker T$ for $k=1,\ldots,m$ and, applying Lemma~\ref{isom} to $T^n$ it follows that
 \begin{equation}\label{mmmcontra2}
\{x_{1,0},\ldots,x_{m,0}\} \mbox{ is a linearly independent set in $\ker T$.}
\end{equation}

Define the index set $\sI$ by
$$
\sI:=\bigl\{(k,j): \ x_{k,j}\notin M, 1\leq k\leq m, 0\leq j\leq n\bigr\}.
$$
The set $\sI$ is non-empty. Otherwise $\{x_{k,0},\ldots x_{k,n}\}\subset M$ for
every $1\leq k\leq m$ and, by Lemma~\ref{Mchain}, these $m$
 (linearly independent) Jordan chains of $T$ at $0$ of length $n+1$
  are as well (linearly independent) Jordan chains of $S$ at $0$ of
  length $n+1$, a contradiction to \eqref{mmm}. Set
$$
h:= \min \bigl\{ j : \  (k,j)\in \sI\,\, \text{for some}\,\, k \,\,
\text{with}\,\, 1\leq k \leq m \bigr\}.
$$
 Without loss of generality, after a reordering of the indices,
assume that $(m,h)\in \sI$, i.e. $x_{m,h}\notin M$. Then,
  \begin{equation}\label{Klaus1}
 j<h \mbox{ implies } x_{k,j}\in M \mbox{ for all } k=1,\ldots, m.
  \end{equation}
 In what follows we construct $m-1$ elements $z_1,\ldots,z_{m-1}$
 in $\ker S^{n+1}$ such that $\{z_1 + \ker S^n,\ldots, z_{m-1} + \ker S^n\}$
  is linearly independent in $\ker S^{n+1} / \ker S^{n}$,
  which is a contradiction to \eqref{mmm}.
  We consider three different cases.

\vskip 0.2cm\noindent
{\bf Case I: $h=n$.}
Since $x_{m,n}\not\in M$, there exist $\alpha_{k,n}\in
\mathbb K$ such that
\[
z_k:= x_{k,n} -\alpha_{k,n} x_{m,n}\in M\cap \ker T^{n+1} \quad
\mbox{for } k=1,\ldots, m-1.
\]
From \eqref{Klaus1} it follows that, for every $k=1, \dots, m-1$, the
 Jordan chain $\{x_{k,0}-\alpha_{k,n}x_{m,0},\ldots, x_{k,n-1}-\alpha_{k,n}x_{m,n-1}, z_k\}$
  of $T$ at $0$ is contained in $M$. Then, by Lemma \ref{Mchain} these are also $m-1$
   (linearly independent) Jordan chains of $S$ at $0$ of length $n$.
%$x_{k,n-1}-\alpha_{k,n}x_{m,n-1},\ldots, x_{k,0}-\alpha_{k,n}x_{m,0}\in M$ for all $k=1, \dots, m-1$.
%Therefore,
%\begin{equation*}
% \begin{split}
%  S^{n+1}z_k&= S^nSz_k =S^n Tz_k\\
%            &=S^n (x_{k,n-1}-\alpha_{k,n}x_{m,n-1})
%            =S^{n-1}S (x_{k,n-1}-\alpha_{k,n}x_{m,n-1})  \\
%  &=S^{n-1}T (x_{k,n-1}-\alpha_{k,n}x_{m,n-1})=S^{n-1} (x_{k,n-2}-\alpha_{k,n}x_{m,n-2})\\
%  &\vdots\\ &=S (x_{k,0}-\alpha_{k,n}x_{m,0})=T (x_{k,0}-\alpha_{k,n}x_{m,0})=0,
% \end{split}
%\end{equation*}
%and $S^n z_k=x_{k,0}-\alpha_{k,n}x_{m,0}\not=0$ for all $k=1, \dots, m-1$.
%The operator
%$S^n$ defined on the space of
%cosets $\spn \{ z_{1}+\ker S^n, \ldots, z_{m-1}+\ker S^n\}$
%is an isomorphism onto $\spn \{x_{1,0}-\alpha_{1,n}x_{m,0},\ldots,x_{m-1,0}-\alpha_{m-1,n}x_{m,0}\}$.
%By \eqref{mmmcontra2} the set $\{x_{1,0}-\alpha_{1,n}x_{m,0},\ldots,x_{m-1,0}-\alpha_{m-1,n}x_{m,0}\}$
%is also linear independent in $\ker S$,
% which shows that $\{z_1 + \ker S^n,\ldots, z_{m-1} + \ker S^n\}$ is linearly independent in $\ker S^{n+1} / \ker S^{n}$.
In particular, the set $\{z_1 + \ker S^n,\ldots, z_{m-1} + \ker S^n\}$ is linearly independent in $\ker S^{n+1} / \ker S^{n}$.

\vskip 0.2cm\noindent
{\bf Case II: $h=n-1$.}  Since $x_{m,n-1}\not\in M$, there
exist $\alpha_{k,n-1}\in \mathbb K$ such that
\[
v_{k,n-1}:= x_{k,n-1} -\alpha_{k,n-1} x_{m,n-1}\in M\cap \ker T^{n} \quad
\mbox{for } k=1,\ldots, m-1.
\]
Let
$
w_{k,n}:= x_{k,n} -\alpha_{k,n-1} x_{m,n}\in \ker T^{n+1}$ for $k=1,\ldots, m-1$
and choose $\alpha_{k,n}\in\mathbb K$ such that
\[
z_k:= w_{k,n}-\alpha_{k,n}x_{m,n-1}\in M\cap \ker T^{n+1} \quad  \mbox{for }\ k=1,\ldots, m-1.
\]
Since $z_k\in M$ and $v_{k,n-1}\in M$, $k=1,\ldots, m-1$, we conclude from $T w_{k,n} = v_{k,n-1}$
together with \eqref{Klaus1} that
\begin{equation*}
 \begin{split}
  S^{n+1}z_k&=S^n Sz_k=S^n Tz_k\\
  &=S^nT(w_{k,n}-\alpha_{k,n}x_{m,n-1})=S^n(v_{k,n-1}- \alpha_{k,n}x_{m,n-2})\\
  &=S^{n-1}T(v_{k,n-1}- \alpha_{k,n}x_{m,n-2})\\
  &=S^{n-1}T(x_{k,n-1} -\alpha_{k,n-1} x_{m,n-1}- \alpha_{k,n}x_{m,n-2})\\
  &=S^{n-1}(x_{k,n-2} -\alpha_{k,n-1} x_{m,n-2}- \alpha_{k,n}x_{m,n-3}) \\
  &\,\,\,\,\vdots\\
  &=S^2(x_{k,1} -\alpha_{k,n-1} x_{m,1}- \alpha_{k,n}x_{m,0})\\
  &=S T (x_{k,1} -\alpha_{k,n-1} x_{m,1}- \alpha_{k,n}x_{m,0})\\
  &=S(x_{k,0} -\alpha_{k,n-1} x_{m,0})=T(x_{k,0} -\alpha_{k,n-1} x_{m,0})=0,\\
 \end{split}
\end{equation*}
and $S^n z_k=x_{k,0}-\alpha_{k,n-1}x_{m,0}\not=0$ for all $k=1, \dots, m-1$.
By \eqref{mmmcontra2} the set $\{x_{1,0}-\alpha_{1,n-1}x_{m,0},\ldots,x_{m-1,0}-\alpha_{m-1,n-1}x_{m,0}\}$
is linearly independent. Then by Lemma~\ref{isom} applied to $S^n$ it follows that the set $\{z_1 + \ker S^n,\ldots, z_{m-1} + \ker S^n\}$
is linearly independent in $\ker S^{n+1} / \ker S^{n}$.

%The operator
%$S^n$ defined on the space of
%cosets $\spn \{ z_{1}+\ker S^n, \ldots, z_{m-1}+\ker S^n\}$
%is a bijection onto $\spn \{x_{1,0}-\alpha_{1,n-1}x_{m,0},\ldots,x_{m-1,0}-\alpha_{m-1,n-1}x_{m,0}\}$.
\vskip 0.2cm\noindent
{\bf Case III: $0\leq h\leq n-2$.}
In this case we construct, as in Case II,
two sets of vectors
\begin{equation}\label{Klaus4}
\bigl\{v_{k,j} \in M\cap \ker T^{j+1} : k=1,\ldots, m-1, j= h,\ldots, n-1\bigr\},
\end{equation}
and
\begin{equation}\label{Klaus5}
\bigl\{w_{k,j+1} \in \ker T^{j+2} : k=1,\ldots, m-1, j= h,\ldots, n-1\bigr\}.
\end{equation}
By assumption, $x_{m,h}\not\in M$. We start the construction with $j=h$, that is, with
 the definition of the vectors $v_{k,h}$
and $w_{k,h+1}$ for $k=1,\ldots, m-1$: There exist $\alpha_{k,h}\in \mathbb K$ such that
\[
v_{k,h}:= x_{k,h} -\alpha_{k,h} x_{m,h}\in M\cap \ker T^{h+1} \quad
\mbox{for } k=1,\ldots, m-1.
\]
Using the same coefficients $\alpha_{k,h}\in \mathbb K$, let
\[
w_{k,h+1}:= x_{k,h+1} -\alpha_{k,h} x_{m,h+1}\in \ker T^{h+2} \quad  \mbox{for }\ k=1,\ldots, m-1.
\]
Notice that $T w_{k,h+1} = v_{k,h}$ for $k=1,\ldots, m-1$.
The vectors $v_{k,j}$
and $w_{k,j+1}$ for $k=1,\ldots, m-1$ are defined inductively
for $j= h+1,\ldots, n-1$, %where one has to start with $j=h+1$ and continue with $j=h+2$ up to $j=n-1$
in the following way:
Fix $j= h+1,\ldots, n-1$ and assume that we have constructed $v_{k,j-1}\in M\cap \ker T^{j}$
and $w_{k,j}\in \ker T^{j+1}$for $k=1,\ldots, m-1$.
Then there exist $\alpha_{k,j}\in\mathbb K$ such that
\[
v_{k,j}:= w_{k,j}-\alpha_{k,j}x_{m,h}\in M\cap \ker T^{j+1}  \quad  \mbox{for }\ k=1,\ldots, m-1.
\]
Also, define
$$
w_{k,j+1}:=x_{k,j+1}-\sum_{i=0}^{j-h} \alpha_{k,h+i}x_{m,j-i+1}\in \ker T^{j+2}
\quad  \mbox{for }\ k=1,\ldots, m-1.
$$
A straightforward computation shows $T w_{k,j+1}= v_{k,j}$ for $k=1,\ldots, m-1$.
So, we have constructed the sets in \eqref{Klaus4} and \eqref{Klaus5}.

Finally, observe that there also exist $\alpha_{k,n}\in\mathbb K$ such that
\[
z_k:= w_{k,n}-\alpha_{k,n}x_{m,h}\in M\cap \ker T^{n+1} \quad  \mbox{for }\ k=1,\ldots, m-1.
\]
Hence,
\begin{align*}
Sz_k  &=  Tz_k=T(w_{k,n}-\alpha_{k,n}x_{m,h}) =v_{k,n-1}-\alpha_{k,n}x_{m,h-1}, \\
S^2z_k  & =  S(v_{k,n-1}-\alpha_{k,n}x_{m,h-1})\\
        &= T(v_{k,n-1}-\alpha_{k,n}x_{m,h-1})\\
        &= T(w_{k,n-1}-\alpha_{k,n-1}x_{m,h} -\alpha_{k,n}x_{m,h-1}) \\
 &=  v_{k,n-2}-\alpha_{k,n-1}x_{m,h-1} -\alpha_{k,n}x_{m,h-2},
\end{align*}
and, in the same way, we show that
\[
S^{n-h}z_k = \ v_{k,h}-\sum_{i=1}^{n-h} \alpha_{k, h+i}x_{m, h-i},
\]
where $x_{m,l}=0$ if $l<0$.
Also, observe that
\begin{align*}
S^{n-h+1}z_k  &=  S(v_{k,h}-\sum_{i=1}^{n-h} \alpha_{k, h+i}x_{m, h-i})\\
&=T(v_{k,h}-\sum_{i=1}^{n-h} \alpha_{k, h+i}x_{m, h-i}) \\
& = T(x_{k,h}-\alpha_{k,h}x_{m,h}-
\sum_{i=1}^{n-h} \alpha_{k, h+i}x_{m, h-i})\\
&= x_{k,h-1}-\sum_{i=0}^{n-h} \alpha_{k, h+i}x_{m, h-i-1}, \\
S^{n-h+2}z_k  & = S(x_{k,h-1}-\sum_{i=0}^{n-h} \alpha_{k, h+i}x_{m, h-i-1})\\
&=
T(x_{k,h-1}-\sum_{i=0}^{n-h} \alpha_{k, h+i}x_{m, h-i-1}) \\ \nonumber
 &= x_{k,h-2}-\sum_{i=0}^{n-h} \alpha_{k, h+i}x_{m, h-i-2}, \\\nonumber
& \,\,\,\,\vdots  \\\label{ArnstadtHbf}
S^n z_k &=  x_{k,0}- \alpha_{k,h} x_{m,0}, \ \ \ \ \text{and} \\ \nonumber
 S^{n+1} z_k &=  0.
\end{align*}
Furthermore, by \eqref{mmmcontra2} the set $\{x_{1,0}-\alpha_{1,h}x_{m,0},\ldots,x_{m-1,0}-\alpha_{m-1,h}x_{m,0}\}$
is linearly independent in $\ker S$. Applying Lemma~\ref{isom} to $S^n$ it follows that the set
$\{ z_{1}+\ker S^n, \ldots, z_{m-1}+\ker S^n\}$ is linearly independent in $\ker S^{n+1}/\ker S^n$.

%Furthermore, the set $\{z_1 + \ker S^n,\ldots, z_{m-1} + \ker S^n\}$ is
%linearly independent in $\ker S^{n+1}/\ker S^n$.
%In fact, by \eqref{mmmcontra2}, the set $\{x_{1,0}-\alpha_{1,h}x_{m,0},\ldots,x_{m-1,0}-\alpha_{m-1,h}x_{m,0}\}$
%is linearly independent in $\ker S$. Then, applying Remark \ref{isom}
%to $S^n$, it follows that
 %$\{ z_{1}+\ker S^n, \ldots, z_{m-1}+\ker S^n\}$ is linearly independent in $\ker S^{n+1}/\ker S^n$.

\vskip 0.2cm
Summing up, we have shown in Cases I-III above
that there exists a linearly independent set $\{ z_{1}+\ker S^n, \ldots, z_{m-1}+\ker S^n\}$
in $\ker S^{n+1}/\ker S^n$, which contradicts \eqref{mmm}.
Therefore,
\begin{equation*}\label{sommer1}
\dim (\ker T^{n+1}/\ker T^n)\leq \dim (\ker S^{n+1}/\ker S^n)+1,
\end{equation*}
and, in particular,
 $\ker T^{n+1} / \ker T^{n}$ is finite dimensional.
Then, \eqref{formel22} follows by interchanging $S$ and $T$.
 %By interchanging $S$ and $T$ we obtain
%\begin{equation*}\label{sommer2}
%\dim (\ker S^{n+1}/\ker S^n)-1\leq \dim (\ker T^{n+1}/\ker T^n),
%\end{equation*}
%and \eqref{formel22} follows.
Finally,  \eqref{formel11} is a consequence of \eqref{RegenErfurt} and
repeated applications of \eqref{formel22}.
\end{proof}

Before proving Theorem \ref{hurra} in Section~\ref{proof} we will improve the upper bound in (ii) of Proposition~\ref{general rank one pert} for a particular class of rank-one perturbations.

Assume that $S$ is a linear operator in $X$ and $M$ is a linear subspace in $\dom S$ such that $\dim\bigl(\dom S/M\bigr)=k$. Then, there exist linearly independent vectors $x_1,\dots,x_k\in (\dom S)\setminus M$ such that
$$
\dom S=M\,\dot +\, \text{span}\{x_1,\dots,x_k\}
$$
We define the restrictions
$$
S_p:=S\upharpoonright \left(M\,\dot +\, \text{span}\{x_1,\dots,x_p\}\right),\qquad 1\leq p\leq k.
$$

\begin{lem}\label{rest}
Given $2\leq p\leq k$, if $\ker S_p^{n+1} / \ker S_p^{n}$ is
  finite dimensional for some $n\in \mathbb N$, then the same holds
for $\ker S_{p-1}^{n+1} / \ker S_{p-1}^{n}$ and
\begin{equation*} \label{Suedkreuz1}
\dim\left(\frac{\ker S_p^{n+1}}{\ker S_p^{n}} \right)-1 \leq \dim \left( \frac{\ker S_{p-1}^{n+1}}{
 \ker S_{p-1}^{n}} \right) \leq \dim\left( \frac{\ker S_p^{n+1}}{\ker S_p^{n}} \right).
 \end{equation*}
\end{lem}

\begin{proof}
By Proposition~\ref{general rank one pert}
only  the second inequality needs to be proved.
Assume that
$ \dim\left( \ker S_p^{n+1} / \ker S_p^{n} \right)=i<\infty$ and that
 the set $\{z_1+\ker S_{p-1}^{n}, \dots, z_{i+1} + \ker S_{p-1}^{n}\}$ is
linearly independent in
$\ker S_{p-1}^{n+1} / \ker S_{p-1}^{n}$. Then, since
$\ker S_{p-1}^{n+1} \subset \ker S_{p}^{n+1}$, there exist $\alpha_1,\dots,\alpha_{i+1}\in\mathbb K$
(not all equal to zero) such that
$$
z:=\alpha_1z_1+\dots +\alpha_{i+1}z_{i+1} \in %\ker S_p^{n+1}\cap
\ker S_p^{n}.
$$
Together with
$z\in\dom S_{p-1}^{n+1}\subset\dom S_{p-1}^n$ we conclude
$z\in\ker S^n_{p-1}$, a contradiction, and Lemma~\ref{rest} is shown.
\end{proof}

\section{Proof of Theorem \ref{hurra}}\label{proof}

We start the proof with some preparations.
By assumption $S$ and $T$ satisfy Hypothesis~{\rm \ref{hypo}}. We discuss the case
$$
\dim\bigl(\dom S/M\bigr)=k\quad\text{and}\quad \dim\bigl(\dom T/M\bigr)=l\leq k.
$$
Then
there exist linearly independent vectors $x_1,\dots,x_k\in (\dom S)\setminus M$ and $y_1,\dots,y_l\in (\dom T)\setminus M$ such that
$$\dom S=M\,\dot +\, \text{span}\{x_1,\dots,x_k\}\quad\text{and}\quad \dom T=M\,\dot +\, \text{span}\{y_1,\dots,y_l\}.$$
Also, we can assume that $\text{span}\{x_1,\dots,x_k\}\cap \text{span}\{y_1,\dots,y_l\}=\{0\}$ (otherwise $M$ can be enlarged).
Next, consider the restrictions
$$
S_p:=S\upharpoonright \left(M\,\dot +\, \text{span}\{x_1,\dots,x_p\}\right),\qquad 1\leq p\leq k,
$$
and
$$
T_q:=T\upharpoonright \left(M\,\dot +\, \text{span}\{y_1,\dots,y_q\}\right),\qquad 1\leq q\leq l.
$$
Clearly $S=S_k$ and $T=T_l$.
As mentioned before, it is sufficient to prove Theorem~\ref{hurra}  for $\lambda=0$. Let $\ker S^{n+1} / \ker S^{n}$ be
  finite dimensional for some $n\in \mathbb N$, $n\geq 1$.
Applying repeatedly Lemma \ref{rest}
to $S=S_k, S_{k-1}, \ldots, S_2$, we see that
$\ker S_1^{n+1} / \ker S_1^{n}$ is finite dimensional
and
\begin{equation}\label{lift1}
\dim\left( \frac{\ker S^{n+1}}{\ker S^{n}} \right)-(k-1) \leq
\dim \left( \frac{\ker S_1^{n+1}}{\ker S_1^{n}}\right) \leq
\dim\left( \frac{\ker S^{n+1}}{\ker S^{n}}\right).
\end{equation}
The operators $S_1$ and $T_1$ satisfy Hypothesis~{\rm \ref{hypo}} with $k=1$.
Hence, by Proposition~\ref{general rank one pert}, 
$\ker T_1^{n+1} / \ker T_1^{n}$ is finite dimensional and
\begin{equation}\label{lift2}
\bigl| \dim \left( \ker S_1^{n+1} / \ker S_1^{n} \right) -
\dim \left( \ker T_1^{n+1} / \ker T_1^{n} \right) \bigr| \leq 1.
\end{equation}
Similarly, repeated application of Lemma \ref{rest}
to $T_2, T_3,\ldots, T_l=T$ shows that
$\ker T^{n+1} / \ker T^{n}$ is finite dimensional and
%another repeated application of Lemma \ref{rest} to $T_2, T_3,\ldots, T_l=T$ yields
\begin{equation}\label{lift3}
\dim\left( \frac{\ker T^{n+1}}{\ker T^{n}} \right)-(l-1) \leq
\dim \left( \frac{\ker T_1^{n+1}}{\ker T_1^{n}}\right) \leq
\dim\left( \frac{\ker T^{n+1}}{\ker T^{n}} \right).
\end{equation}
Since $l\leq k$, notice that $-(k-1)\leq -(l-1)$. Therefore
with \eqref{lift1},  \eqref{lift2} and \eqref{lift3}
\begin{equation*}
 \begin{split}
& \dim \left( \ker S^{n+1} / \ker S^{n} \right) - \dim \left( \ker T^{n+1} / \ker T^{n} \right) \\
& \qquad\qquad \geq \dim \left( \ker S_1^{n+1} / \ker S_1^{n} \right)- \dim \left( \ker T^{n+1} / \ker T^{n} \right) \\
& \qquad\qquad \geq \dim \left( \ker T_1^{n+1} / \ker T_1^{n} \right) - 1 - \dim \left( \ker T^{n+1} / \ker T^{n} \right) \\
& \qquad\qquad \geq -(l-1) -1 \\
& \qquad\qquad \geq -(k-1) -1 = -k.
\end{split}
\end{equation*}
An analogous calculation for the upper bound shows
$$
\dim \left( \ker S^{n+1} / \ker S^{n} \right) - \dim \left( \ker T^{n+1} / \ker T^{n} \right)\leq  k,
$$
%Analogously, with \eqref{lift3},  \eqref{lift2} and \eqref{lift1},
%\begin{equation*}
% \begin{split}
%& \dim \left( \ker S^{n+1} / \ker S^{n} \right) - \dim \left( \ker T^{n+1} / \ker T^{n} \right) \\
%& \qquad \leq \dim \left( \ker S^{n+1} / \ker S^{n} \right)- \dim \left( \ker T_1^{n+1} / \ker T_1^{n} \right) \\
%& \qquad \leq \dim \left( \ker S^{n+1} / \ker S^{n} \right) + 1 - \dim \left( \ker S_1^{n+1} / \ker S_1^{n} \right) \\
%& \qquad \leq 1+ (k-1) = k,
%\end{split}
%\end{equation*}
which yields
$$
\bigl| \dim \left( \ker S^{n+1} / \ker S^{n} \right) -
\dim \left( \ker T^{n+1} / \ker T^{n} \right) \bigr| \leq k,
$$
and assertion (ii) in Theorem~\ref{hurra} holds.
Finally, assertion (i) in Theorem~\ref{hurra} follows from
$$
|\dim\ker S - \dim \ker T |\leq k,
$$
which is shown in a similar
way as in the proof of Proposition~\ref{general rank one pert},
and a repeated application of \eqref{formel2}.

% *******************************************************************

\bibliographystyle{amsplain}

\subsection*{Contact information}

{\bf Jussi Behrndt}

Institut f\"ur Numerische Mathematik, Technische Universit\"at Graz

Steyrergasse 30, 8010 Graz, Austria

behrndt@tugraz.at

\noindent
{\bf Leslie Leben}

Institut f\"ur  Mathematik,  Technische Universit\"{a}t Ilmenau

Postfach 100565, D-98684 Ilmenau,  Germany

leslie.leben@tu-ilmenau.de

\noindent
{\bf Francisco Mart\'{\i}nez Per\'{\i}a }

Depto. de Matem\'{a}tica, Fac. Cs. Exactas, Universidad Nacional de La Plata

C.C.\ 172, (1900) La Plata, Argentina

and Instituto Argentino de Matem\'{a}tica "Alberto P. Calder\'{o}n" (CONICET)

Saavedra 15, (1083) Buenos Aires, Argentina

francisco@mate.unlp.edu.ar

\noindent
{\bf Carsten Trunk }

Institut f\"ur  Mathematik,  Technische Universit\"{a}t Ilmenau

Postfach 100565, D-98684 Ilmenau,  Germany

carsten.trunk@tu-ilmenau.de

\end{document}